\documentclass[12pt,reqno]{amsart}
\usepackage{amsmath, amsthm, amssymb, amsbsy}

\advance \topmargin by -\headheight
\advance \topmargin by -\headsep
     
\evensidemargin \oddsidemargin
\newtheorem{theorem}{Theorem}[section]
\newtheorem{lemma}[theorem]{Lemma}

\theoremstyle{definition}

\theoremstyle{remark}

\numberwithin{equation}{section}

\def\bfc{{\mathbf c}}
\def\bfd{{\mathbf d}}

\def\bfn{{\mathbf n}}

\def\calA{{\mathcal A}}  

\def\calB{{\mathcal B}}

\def\calF{{\mathcal F}}

\def\calZ{{\mathcal Z}}

\def\Etil{\widetilde{E}}

\def\dbN{{\mathbb N}}

\def\dbZ{{\mathbb Z}}

\def\grm{{\mathfrak m}}\def\grM{{\mathfrak M}}\def\grN{{\mathfrak N}}
\def\grn{{\mathfrak n}}\def\grS{{\mathfrak S}}

\def\alp{{\alpha}} \def\bet{{\beta}}
\def\gam{{\gamma}} \def\Gam{{\Gamma}}
\def\del{{\delta}}   
\def\tet{{\theta}} 
\def\Tet{{\Theta}}

\def\lam{{\lambda}}  
\def\bflam{{\boldsymbol \lambda}}

\def\ome{{\omega}} 

\def\eps{\varepsilon}

\def\le{\leqslant} \def\ge{\geqslant}

\begin{document}
\title[Exceptional sets]{Davenport's method and slim exceptional sets:\\ the 
asymptotic formulae in Waring's problem}
\author[K. Kawada]{Koichi Kawada}
\address{KK: Department of Mathematics, Faculty of Education, Iwate University, 
Morioka, 020-8550 Japan}
\email{kawada@iwate-u.ac.jp}
\author[T. D. Wooley]{Trevor D. Wooley$^*$}
\address{TDW: School of Mathematics, 
University of Bristol, 
University Walk, Clifton, 
Bristol BS8 1TW, United Kingdom}
\email{matdw@bristol.ac.uk}
\thanks{$^*$The second author is supported by a Royal Society Wolfson Research
 Merit Award.}
\date{}
\begin{abstract} We apply a method of Davenport to improve several estimates for
 slim exceptional sets associated with the asymptotic formula in Waring's 
problem. In particular, we show that the anticipated asymptotic formula in 
Waring's problem for sums of seven cubes holds for all but $O(N^{1/3+\eps})$ of 
the natural numbers not exceeding $N$.\end{abstract}
\maketitle

\section{Introduction} Earlier work concerning slim exceptional sets in Waring's
 problem is based on the introduction of an exponential sum over the exceptional
 set, and a subsequent analysis of auxiliary mean values involving the latter 
generating function (see \cite{slim1}, \cite{slim2}, \cite{slim3}, \cite{slim4},
 \cite{slim5}). Such a strategy replaces the application of Bessel's inequality 
conventionally applied within the Hardy-Littlewood (circle) method. Loosely 
speaking, the newer methods show that when exceptional sets are small, then they
 are necessarily {\it very} small, and an obstruction to further progress is 
the difficulty of establishing the former prerequisite. An old method of 
Davenport \cite{Dav1942} is based on a Diophantine interpretation of the 
application of Cauchy's inequality restricted to thin sequences. Our goal in 
this paper is to show how Davenport's method may be applied to good effect in 
deriving slim exceptional set estimates, thereby expanding the catalogue of 
problems accessible to slim technology.\par

When $s$ and $k$ are natural numbers, we denote by $R_{s,k}(n)$ the number of 
representations of a positive integer $n$ as the sum of $s$ $k$th powers of 
positive integers. A heuristic application of the circle method suggests that 
for $k\ge 3$ and $s\ge k+1$, one should have the asymptotic relation
\begin{equation}\label{1.1}
R_{s,k}(n)=\frac{\Gam (1+1/k)^s}{\Gam(s/k)}\grS_{s,k}(n)n^{s/k-1}+o(n^{s/k-1}),
\end{equation}
where
$$\grS_{s,k}(n)=\sum_{q=1}^\infty \sum^q_{\substack{a=1\\ (a,q)=1}}\Bigl( q^{-1}\sum_{r=1}^q
e(ar^k/q)\Bigr)^se(-an/q),$$
and $e(z)$ denotes $\exp(2\pi iz)$. It is worth noting here that, under modest 
congruence conditions, one has $1\ll \grS_{s,k}(n)\ll n^\eps$, and thus the 
conjectural relation (\ref{1.1}) may be interpreted as an honest asymptotic 
formula (see sections 4.3, 4.5 and 4.6 of \cite{Vau1997} for details). We 
measure the frequency with which the formula (\ref{1.1}) fails by defining an 
associated exceptional set as follows. When $\psi(t)$ is a function of a 
positive variable $t$, we denote by $\Etil_{s,k}(N;\psi)$ the number of integers 
$n$, with $1\le n\le N$, for which
\begin{equation}\label{1.2a}
\left| R_{s,k}(n)-\frac{\Gam(1+1/k)^s}{\Gam(s/k)}\grS_{s,k}(n)n^{s/k-1}\right|>
n^{s/k-1}\psi(n)^{-1}.
\end{equation}

\par By applying classical methods based on the use of Bessel's inequality, one 
may derive from work of Vaughan \cite{Vau1986a} and \cite{Vau1986b} the estimate
\begin{equation}\label{1.3}
\Etil_{s,k}(N;\psi)\ll N^{1-(s2^{2-k}-2)/k}(\log N)^{-\nu}\psi(N)^2\quad (2^{k-1}\le 
s\le 2^k),
\end{equation}
valid whenever $\psi(N)$ grows sufficiently slowly, with the exponent 
$\nu=\nu(s,k)$ positive when $s=2^{k-1}$. This work also establishes that when 
$\psi(t)$ grows no faster than a suitable power of $\log t$, then 
$\Etil_{s,k}(N;\psi)\ll 1$ for $s\ge 2^k$. In \S2 we improve on the upper bound 
(\ref{1.3}) whenever $s>\frac{3}{4}2^k$. For ease of future reference, we 
summarise our new conclusions followed by those previously available separately 
for each exponent $k$. It is convenient here, and in what follows, to refer to 
a function $\psi(t)$ as being a {\it sedately increasing function} when 
$\psi(t)$ is a function of a positive variable $t$, increasing monotonically to 
infinity, and satisfying the condition that when $t$ is large, one has 
$\psi(t)=O(t^\del)$ for a positive number $\del$ sufficiently small in the 
ambient context. 

\begin{theorem}\label{theorem1.1} Suppose that $\psi_3(t)$ is a sedately 
increasing function. Then for each $\eps>0$, one has $\Etil_{7,3}(N;\psi_3)\ll 
N^{1/3+\eps}\psi_3(N)^4$.
\end{theorem}

For comparison, the relation (\ref{1.3}), which in this case yields a bound of 
quality $\Etil_{s,3}(N;\psi)\ll N^{1-(s-4)/6}\psi_3(N)^2$ $(4\le s\le 7)$, supplies 
the estimate $\Etil_{7,3}(N;\psi_3)\ll N^{1/2}\psi_3(N)^2$. This was improved in 
Theorem 1.3 of \cite{slim2}, so that whenever $\psi(t)=O((\log t)^{1-\del})$ for 
some $\del>0$, then $\Etil_{7,3}(N;\psi)\ll N^{4/9+\eps}$. The conclusion of 
Theorem \ref{theorem1.1} is superior to both estimates.

\begin{theorem}\label{theorem1.2} Suppose that $\psi_4(t)$ is a sedately 
increasing function. Then, for each $\eps>0$, one has $\Etil_{s,4}(N;\psi_4)\ll 
N^{\alp_s+\eps}\psi_4(N)^4$ $(s=13,14)$, where $\alp_{13}=\frac{5}{8}$ and 
$\alp_{14}=\frac{1}{2}$.
\end{theorem}

The conclusion of Theorem 1.1 of \cite{slim5} provides a bound of the same shape
 as that supplied by this theorem when $s=15$, though with 
$\alp_{15}=\frac{7}{16}$ and the factor $\psi_4(N)^4$ replaced by $\psi_4(N)^2$. 
Meanwhile, the earlier bound (\ref{1.3}) in this case yields an estimate of the 
latter type with $\alp_s=1-\frac{s-8}{16}$ $(8\le s\le 16)$.

\begin{theorem}\label{theorem1.3}
Suppose that $\psi_5(t)$ is a sedately increasing function. Then, for each 
$\eps>0$, one has $\Etil_{s,5}(N;\psi_5)\ll N^{\bet_s+\eps}\psi_5(N)^4$ $(25\le s\le 
28)$, where $\bet_s=\frac{4}{5}-\frac{s-24}{20}$.
\end{theorem}

A bound of the same type is supplied by Theorem 1.2 of \cite{slim5}, though with
 $\bet_{29}=\frac{23}{40}$, $\bet_{30}=\frac{11}{20}$ and $\bet_{31}=\frac{3}{8}$, 
and the factor $\psi_5(N)^4$ replaced by $\psi_5(N)^2$. Meanwhile, the estimate 
(\ref{1.3}) yields analogous bounds with $\bet_s=1-\frac{s-16}{40}$ 
$(16\le s\le 32)$.

\begin{theorem}\label{theorem1.4}
Suppose that $\psi_6(t)$ is a sedately increasing function. Then, for each $\eps
>0$, one has $\Etil_{s,6}(N;\psi_6)\ll N^{\gam_s+\eps}\psi_6(N)^4$ $(44\le s\le 51)$,
 where $\gam_s=\frac{5}{6}-\frac{s-43}{48}$.
\end{theorem}

An estimate of this shape is delivered by Theorem 1.3 of \cite{slim5}, though 
with $\gam_s=\frac{2}{3}-\frac{s-51}{96}$ $(52\le s\le 55)$, and the factor 
$\psi_6(N)^4$ replaced by $\psi_6(N)^2$. Meanwhile, a careful application of the 
methods of Heath-Brown \cite{HB1988} and Boklan \cite{Bok1994} yields bounds of 
this type with $\gam_s=1-\frac{s-28}{72}$ $(28\le s\le 31)$ and 
$\gam_s=1-\frac{s-27}{96}$ $(32\le s\le 55)$. The main conclusion of 
\cite{Bok1994}, in particular, shows that $\Etil_{s,6}(N;\psi_6)\ll 1$ when 
$s\ge 56$, provided that $\psi_6(t)=O((\log t)^\del)$ for a suitably small 
positive number 
$\del$. 

\begin{theorem}\label{theorem1.5}
Suppose that $\psi_7(t)$ is a sedately increasing function. Then, for each 
$\eps>0$, one has $\Etil_{s,7}(N;\psi_7)\ll N^{\del_s+\eps}\psi_7(N)^4$ $(85\le s\le 
100)$, where $\del_s=\frac{6}{7}-\frac{s-84}{112}$.
\end{theorem}

For comparison, Theoren 1.4 of \cite{slim5} delivers an estimate of this shape 
with $\del_s=\frac{5}{7}-\frac{s-100}{224}$ $(101\le s\le 108)$ and 
$\del_s=\frac{4}{7}-\frac{s-108}{224}$ $(109\le s\le 111)$, though with the 
factor $\psi_7(N)^4$ again replaced by $\psi_7(N)^2$. Meanwhile, the methods of 
\cite{HB1988} and \cite{Bok1994} yield bounds of this type with 
$\del_s=1-\frac{s-56}{168}$ $(56\le s\le 68)$ and $\del_s=1-\frac{s-52}{224}$ 
$(69\le s\le 111)$. Also, one finds from \cite{Bok1994} that $\Etil_{s,7}(N;
\psi_7)\ll 1$ for $s\ge 112$, provided that $\psi_7(t)=O((\log t)^\del)$ and 
$\del>0$ is suitably small.

\begin{theorem}\label{theorem1.6}
Suppose that $\psi_8(t)$ is a sedately increasing function. Then, for each 
$\eps>0$, one has $\Etil_{s,8}(N;\psi_8)\ll N^{\eta_s+\eps}\psi_8(N)^4$, where 
$\eta_s=\frac{7}{8}-\frac{s-168}{192}$ $(171\le s\le 180)$, and 
$\eta_s=\frac{7}{8}-\frac{s-164}{256}$ $(181\le s\le 196)$.
\end{theorem}

The conclusion of Theorem 1.5 of \cite{slim5} in this instance delivers an 
estimate of the above shape with $\eta_s=\frac{3}{4}-\frac{s-196}{512}$ 
$(197\le s\le 212)$, $\eta_s=\frac{5}{8}-\frac{s-212}{512}$ $(213\le s\le 220)$,
 and $\eta_s=\frac{1}{2}-\frac{s-220}{512}$ $(221\le s\le 223)$, though again 
with the factor $\psi_8(N)^4$ replaced by $\psi_8(N)^2$. Meanwhile, the methods 
of \cite{HB1988} and \cite{Bok1994} may be wrought to provide estimates of the 
latter type with $\eta_s=1-\frac{s-112}{384}$ $(112\le s\le 148)$ and 
$\eta_s=1-\frac{s-100}{512}$ $(149\le s\le 223)$. In addition, the bound 
$\Etil_{s,8}(N;\psi_8)\ll 1$, for $s\ge 224$, follows from \cite{Bok1994} 
provided that $\psi_8(t)=O((\log t)^\del)$ for a suitably small positive number 
$\del$.\par

When $k\ge 9$, our methods are again of use in estimating $\Etil_{s,k}(N;\psi_k)$
 when $s$ is relatively large, though Vinogradov's methods increasingly dominate
 the analysis and transform the landscape (see \cite{For1995} and \cite{Woo1992}
 for the relevant ideas). We therefore avoid discussion of the situation for 
larger values of $k$.\par

As our next application of Davenport's method interpreted through slim 
technology, we consider higher moments of the counting functions $R_{s,k}(n)$. 
In order to illustrate ideas, we concentrate on sums of cubes, and in \S3 
derive an improvement on recent work of Br\"udern and the second author 
\cite{BW2009}.

\begin{theorem}\label{theorem1.7}
For any positive number $h$ smaller than $\frac{7}{2}$, there is a positive 
number $\del=\del(h)$ with the property that
$$\sum_{1\le n\le N}\Bigl| R_{5,3}(n)-\frac{\Gam({\textstyle{\frac{4}{3}}})^5}
{\Gam({\textstyle{\frac{5}{3}}})}\grS_{5,3}(n)n^{2/3}\Bigr|^h\ll N^{2h/3+1-\del}.$$
In addition, for each $\eps>0$, one has
\begin{equation}\label{1.10}
\sum_{1\le n\le N}\Bigl| R_{5,3}(n)-\frac{\Gam({\textstyle{\frac{4}{3}}})^5}
{\Gam({\textstyle{\frac{5}{3}}})}\grS_{5,3}(n)n^{2/3}\Bigr|^3\ll N^{35/12+\eps}.
\end{equation}
\end{theorem}

The first conclusion of this theorem includes Theorem 1.1 of \cite{BW2009} as 
the special case in which $h=3$. The second estimate, meanwhile, has the same 
strength as Theorem 1.2 of \cite{BW2009}, in which it is asserted that
$$\sum_{1\le n\le N}\Bigl( R_{5,3}(n)-\frac{\Gam({\textstyle{\frac{4}{3}}})^5}
{\Gam({\textstyle{\frac{5}{3}}})}\grS_{5,3}(n)n^{2/3}\Bigr)^3\ll N^{35/12+\eps}.$$
Our conclusion here has greater content, and also supersedes the conclusion of 
Theorem 1.1 of \cite{BW2009}, in which the bound (\ref{1.10}) is obtained with 
the right hand side replaced by $N^3(\log N)^{\eps-4}$.\par

Before announcing our final application, we require some additional notation. 
When $P$ and $R$ are real numbers with $1\le R\le P$, we define the set of 
smooth numbers $\calA(P,R)$ by
$$\calA(P,R)=\{ n\in [1,P]\cap \dbZ:\text{$p$ prime and $p|n\Rightarrow p\le R$}
\}.$$
We then define the exponential sum $h(\alp)=h(\alp;P,R)$ by
$$h(\alp;P,R)=\sum_{y\in \calA(P,R)}e(\alp y^3).$$
The sixth moment of the latter sum has played an important role in a plethora of
 recent applications. Write $\tau=(213-4\sqrt{2833})/164=1/1703.6\dots$. Then 
as a consequence of the work of the second author \cite{Woo2000}, given any 
$\eps>0$, there exists a positive number $\eta=\eta(\eps)$ with the property 
that whenever $1\le R\le P^\eta$, one has
\begin{equation}\label{1.11}
\int_0^1|h(\alp;P,R)|^6\,d\alp \ll P^{13/4-\tau+\eps}.
\end{equation}
We henceforth assume that whenever $R$ appears in a statement, either implicitly
 or explicitly, then $1\le R\le P^\eta$ with $\eta $ a positive number 
sufficiently small in the context of the upper bound (\ref{1.11}). Finally, 
when $\calB\subseteq [P/2,P]\cap \dbZ$, we define the exponential sum $F(\alp)=
F(\alp;\calB)$ by
$$F(\alp;\calB)=\sum_{x\in \calB}e(\alp x^3).$$

\begin{theorem}\label{theorem1.8}
Suppose that $c_i,d_i$ $(i=1,2,3)$ are integers satisfying the condition
$$(c_1d_2-c_2d_1)(c_1d_3-c_3d_1)(c_2d_3-c_3d_2)\ne 0.$$
Write $\lam_j=c_j\alp+d_j\bet$ $(j=1,2,3)$. Also, let $\calB\subseteq [P/2,P]
\cap \dbZ$. Then for each $\eps>0$, there exists a positive number 
$\eta=\eta(\eps)$ such that, whenever $1\le R\le P^\eta$, one has the estimate
$$\int_0^1\int_0^1\prod_{i=1}^3|F(\lam_i)^2h(\lam_i)^2|\,d\alp\,d\bet \ll P^{49/8-
3\tau/2+\eps}.$$
In addition, one has
$$\int_0^1\int_0^1|h(\lam_1)h(\lam_2)h(\lam_3)|^4\,d\alp\,d\bet \ll P^{49/8-
3\tau/2+\eps}.$$
\end{theorem}

For comparison, Theorem 4 of \cite{BW2007} contains a similar conclusion to the 
second estimate of Theorem \ref{theorem1.8}, save that our exponent 
$\frac{49}{8}-\frac{3}{2}\tau$ is there replaced by $\frac{25}{4}-\tau$. The 
twelfth moment estimate supplied by Theorem 4 of \cite{BW2007} was employed, 
together with its close kin, so as to establish the validity of the Hasse 
principle for pairs of diagonal cubic equations in thirteen or more variables. 
We are not aware of additional applications stemming from Theorem 
\ref{theorem1.8}, though quantitative improvements in potential effective 
bounds for solutions ought to benefit from our sharper estimate.\par

Throughout, the letter $\eps$ will denote a sufficiently small positive number. 
We use $\ll$ and $\gg$ to denote Vinogradov's well-known notation, implicit 
constants depending at most on $\eps$, unless otherwise indicated. In an effort 
to simplifiy our analysis, we adopt the convention that whenever $\eps$ appears 
in a statement, then we are implicitly asserting that for each $\eps>0$, the 
statement holds for sufficiently large values of the main parameter. Note that 
the ``value'' of $\eps$ may consequently change from statement to statement, 
and hence also the dependence of implicit constants on $\eps$. Finally, from 
time to time we make use of vector notation in order to save space. Thus, for 
example, we may abbreviate $(c_1,\dots ,c_t)$ to $\bfc$.    

\section{The asymptotic formula in Waring's problem} Our initial approach to the
 task of proving the first six theorems follows closely that taken in the 
second author's previous work \cite{slim5}. Initially, we consider integers 
$k$ and $s$ with $3\le k\le 8$ and $s\ge \frac{3}{4}2^k$. Suppose that $N$ is a 
large positive number, and let $\psi=\psi_k(t)$ be a sedately increasing 
function. We denote by $\calZ_{s,k}(N)$ the set of integers $n$ with $N/2<n\le N$
 for which the inequality (\ref{1.2a}) holds, and we abbreviate 
$\text{card}(\calZ_{s,k}(N))$ to $Z=Z_{s,k}$.\par

Write $P=P_k$ for $[N^{1/k}]$, and define the exponential sum $f(\alp)=f_k(\alp)$
 by
$$f_k(\alp)=\sum_{1\le x\le P_k}e(\alp x^k).$$
Also, let $\grM=\grM_k$ denote the union of the intervals
$$\grM_k(q,a)=\{ \alp \in [0,1): |q\alp -a|\le (2k)^{-1}P_kN^{-1}\},$$
with $0\le a\le q\le (2k)^{-1}P_k$ and $(a,q)=1$, and define $\grm=\grm_k$ by 
putting $\grm_k=[0,1)\setminus \grM_k$. Then the argument of \cite{slim5} 
leading to equation (2.5) reveals that there exist complex numbers 
$\eta_n=\eta_n(s,k)$ with $|\eta_n|=1$, satisfying the condition that, with the 
exponential sum $K(\alp)=K_{s,k}(\alp)$ defined by
$$K_{s,k}(\alp)=\sum_{N/2<n\le N}\eta_n(s,k)e(n\alp),$$
one has
\begin{equation}\label{2.1}
\int_\grm |f_k(\alp)^sK_{s,k}(\alp)|\,d\alp \gg N^{s/k-1}\psi_k(N)^{-1}Z_{s,k}.
\end{equation}
Our goal is now to obtain an upper bound for the integral on the left hand side 
of the relation (\ref{2.1}), and thereby obtain an upper bound on $Z_{s,k}$. This
 we achieve by exploiting an estimate whose roots go back to a method of 
Davenport \cite{Dav1942}.

\begin{lemma}\label{lemma2.1}
Let $k$ be a natural number with $k\ge 3$, and suppose that $1\le j\le k-2$. 
Then one has
$$\int_0^1|f_k(\alp)^{2^j}K_{s,k}(\alp)^2|\,d\alp \ll P_k^{2^j-1}Z_{s,k}+
P_k^{2^j-j/2-1+\eps}Z_{s,k}^{3/2}.$$
\end{lemma}

\begin{proof} The claimed estimate is immediate from the conclusion of Lemma 
6.1 of \cite{KW2009}.\end{proof}

Suppose now that $l$ is a natural number, and put $s=\frac{3}{4}2^k+l$. Then an 
application of Schwarz's inequality shows that
\begin{align}
\int_\grm |f(\alp)^sK(\alp)|\,d\alp \le &\,\left( \sup_{\alp \in \grm}|f(\alp)|
\right)^{2^{k-3}+l}\Bigl( \int_0^1|f(\alp)|^{2^k}\,d\alp \Bigr)^{1/2}\notag \\
&\, \times\Bigl( \int_0^1|f(\alp)^{2^{k-2}}K(\alp)^2|\,d\alp \Bigr)^{1/2}.
\label{2.2}
\end{align}
But Weyl's inequality (see, for example, Lemma 2.4 of \cite{Vau1997}) yields the
 upper bound
$$\sup_{\alp \in \grm}|f(\alp)|\ll P^{1-2^{1-k}+\eps},$$
and Hua's lemma (see Lemma 2.5 of \cite{Vau1997}) supplies the estimate
$$\int_0^1|f(\alp)|^{2^k}\,d\alp \ll P^{2^k-k+\eps}.$$
Consequently, applying these estimates in combination with the case $j=k-2$ of 
Lemma \ref{lemma2.1}, we deduce from (\ref{2.2}) that
\begin{align*}
\int_\grm |f(\alp)^sK(\alp)|\,d\alp \ll&\, P^\eps (P^{1-2^{1-k}})^{2^{k-3}+l}(P^{2^k-k})^{
1/2}\\
&\, \times (P^{2^{k-2}-1}Z+P^{2^{k-2}-k/2}Z^{3/2})^{1/2}\\
\ll &\, P^{s-k-l2^{1-k}+\eps}(P^{k-3/2}Z+P^{(k-1)/2}Z^{3/2})^{1/2}.
\end{align*}
Substituting this bound into (\ref{2.1}), we find that
$$Z\ll \psi_k(N)P^{\eps -l2^{1-k}}(P^{k-3/2}Z+P^{(k-1)/2}Z^{3/2})^{1/2},$$
whence
$$Z\ll P^{k-3/2-l2^{2-k}+\eps}\psi_k(N)^2+P^{k-1-l2^{3-k}+\eps}\psi_k(N)^4.$$
Since $P\asymp N^{1/k}$, we conclude that
$$Z\ll N^{1-(3+l2^{3-k})/(2k)+\eps}\psi_k(N)^2+N^{1-(1+l2^{3-k})/k+\eps}\psi_k(N)^4.$$
In particular, when $1\le l\le 2^{k-3}$, one obtains
\begin{equation}\label{2.3}
Z\ll N^{1-(1+l2^{3-k})/k+\eps}\psi_k(N)^4.
\end{equation}

Recall that $s=\frac{3}{4}2^k+l$. Then on summing over dyadic intervals, one 
finds that the bound (\ref{2.3}) leads to the estimate
$$\Etil_{s,k}(N;\psi)\ll N^{1-(s2^{3-k}-5)/k+\eps}\psi_k(N)^4=N^{\ome_{s,k}+\eps}\psi(N)^4,$$
in which
$$\ome_{s,k}=1-\frac{1}{k}-\frac{s-{\textstyle{\frac{3}{4}}}2^k}{k2^{k-3}}.$$
This confirms the estimates claimed in Theorems \ref{theorem1.1}, 
\ref{theorem1.2} and \ref{theorem1.3}.\medskip

When $k\ge 6$, improvements may be wrought via the technology introduced by 
Heath-Brown \cite{HB1988}, and refined by Boklan \cite{Bok1994}. The key 
elements of such an approach, so far as the application at hand is concerned, 
are contained in the following lemma. In this context, when $r$ is a 
non-negative integer, we write
$$\Tet_{r,k}=P^{k-2r}\int_0^1|f(\alp)^{2r}K(\alp)^2|\,d\alp .$$

\begin{lemma}\label{lemma2.2} Suppose that $k\ge 6$, and that $s,t,u,v$ are 
non-negative integers with
$$s={\textstyle{\frac{7}{16}}}2^k+t+u\quad \text{and}\quad s={\textstyle{\frac
{3}{8}}}2^k+[(k+1)/2]+u+v.$$
Then for each $\eps>0$, one has
$$\int_\grm|f_k(\alp)^sK_{s,k}(\alp)|\,d\alp \ll P_k^{s-k+\eps}(P_k^{-2^{3-k}t/3}+
P_k^{-2^{1-k}v})\Tet_{u,k}^{1/2}.$$
\end{lemma}

\begin{proof} The conclusion of the lemma is an immediate consequence of Lemma 
4.1 of \cite{slim5}.\end{proof}

It transpires that the methods of this paper do not supersede the classical 
bounds reported in the introduction when $s\le s_0$, where
$$s_0={\textstyle{\frac{5}{8}}}2^k+[(k+1)/2].$$
We therefore restrict attention henceforth to the situation with $s>s_0$. We 
apply Lemma \ref{lemma2.2} with $u=2^{k-3}$, $t=2^{k-4}+[(k+1)/2]+l$ for some 
$l\ge 1$, and $v=2^{k-4}+t-[(k+1)/2]$. Observe that one then has
$$s=s_0+l={\textstyle{\frac{9}{16}}}2^k+t=2^{k-1}+[(k+1)/2]+v.$$
Define $l_0=l_0(k)$ by
$$l_0(k)=\begin{cases}
1,&\text{when $k=6,7$,}\\
17,&\text{when $k=8$.}
\end{cases}$$
Then a modest computation reveals that when $l\ge l_0$, one has
$${\textstyle{\frac{2}{3}}}t\ge {\textstyle{\frac{2}{3}}}(2^{k-4}+[(k+1)/2]+l_0)
>2^{k-3}-2[(k+1)/2],$$
whence
$${\textstyle{\frac{8}{3}}}2^{-k}t>2^{1-k}(2^{k-4}+t-[(k+1)/2]).$$
It follows that when $l\ge l_0$, one has $\frac{8}{3}2^{-k}t>2^{1-k}v$, and that 
the reverse inequality holds only when $1\le l<l_0$. In particular, we deduce 
from Lemma \ref{lemma2.2} that
$$\int_\grm |f(\alp)^sK(\alp)|\,d\alp \ll P^{s-k+\eps}(P^{-w_s})\Tet_{u,k}^{1/2},$$
where $w_s=\frac{8}{3}2^{-k}t$ when $1\le l<l_0$, and $w_s=2^{1-k}v$ when $l\ge 
l_0$.\par

We first examine the situation in which $l\ge l_0$, so that $w_s=2^{1-k}v$. Here,
 employing Lemma \ref{lemma2.1} with $j=k-2$ in order to estimate $\Tet_{u,k}$, 
we infer from 
(\ref{2.1}) that
$$N^{s/k-1}\psi_k(N)^{-1}Z\ll P^{s-k+\eps}(P^{-w_s})(P^{k-1}Z+P^{k/2}Z^{3/2})^{1/2}.$$
Thus, on recalling that $P=[N^{1/k}]$ and $v=s-2^{k-1}-[(k+1)/2]$, we deduce that
\begin{align}
Z&\ll P^{k-1-2^{2-k}v+\eps}\psi_k(N)^2+P^{k-2^{3-k}v+\eps}\psi_k(N)^4\notag \\
&\ll N^{a(s,k)+\eps}\psi_k(N)^2+N^{b(s,k)+\eps}\psi_k(N)^4,\label{2.w}
\end{align}
where
$$a(s,k)=1-\frac{1}{k}-\frac{s-2^{k-1}-[(k+1)/2]}{k2^{k-2}}$$
and
$$b(s,k)=1-\frac{1}{k}-\frac{s-5\cdot 2^{k-3}-[(k+1)/2]}{k2^{k-3}}.$$
Write
$$s_1={\textstyle{\frac{3}{4}}}2^k+[(k+1)/2].$$
Then one may verify that $b(s,k)\ge a(s,k)$ when $s\le s_1$, and otherwise 
$b(s,k)<a(s,k)$. Hence, by summing over dyadic intervals, we conclude that
$$\Etil_{s,k}(N)\ll \begin{cases} N^{\ome_{s,k}+\eps}\psi_k(N)^4,&\text{when $s_0<s\le
 s_1$,}\\
N^{\ome_{s,k}+\eps}\psi_k(N)^2,&\text{when $s>s_1$,}
\end{cases}
$$
where
$$\ome_{s,k}=1-\frac{1}{k}-\frac{s-s_0}{k2^{k-3}},$$
when $s_0<s\le s_1$, and
$$\ome_{s,k}=1-\frac{2}{k}-\frac{s-s_1}{k2^{k-2}},$$
when $s>s_1$. This confirms the upper bounds asserted in Theorems 
\ref{theorem1.4} and \ref{theorem1.5}, and also that asserted in Theorem 
\ref{theorem1.6} when $s\ge 181$.\par

We turn now to the complementary situation in which $k=8$ and $1\le l<l_0$. 
In this case we have $w_s=\frac{1}{96}t$. Noting the adjustment in the value of 
$w_s$, and recalling that $t=s-144$, we may proceed exactly as above to obtain 
the estimate (\ref{2.w}), though now with
$$a(s,8)=\frac{7}{8}-\frac{s-144}{384}$$
and
$$b(s,8)=\frac{7}{8}-\frac{s-168}{192}.$$
One may confirm that $b(s,8)\ge a(s,8)$ when $s\le 192$, and otherwise 
$b(s,8)<a(s,8)$. Hence, by summing over dyadic intervals, we conclude that
$$\Etil_{s,8}(N)\ll N^{\ome_{s,8}+\eps}\psi_8(N)^4,$$
when $164<s\le 180$, where $\ome_{s,8}=\frac{7}{8}-\frac{s-168}{192}$. Note that 
when $s\le 170$, the latter bound is weaker than what follows by appropriate 
application of the methods of \cite{HB1988} and \cite{Bok1994} (see the 
discussion in the introduction following the statement of Theorem 
\ref{theorem1.6}). In this way, we have confirmed the upper bound asserted in 
Theorem \ref{theorem1.6} for $171\le s\le 180$.\par

\section{Sums of five cubes} Our goal in this section is the proof of Theorem 
\ref{theorem1.7}, and to this end we adapt the treatment of \S3 of Br\"udern 
and Wooley \cite{BW2009} so as to incorporate the estimate supplied by Lemma 
\ref{lemma2.1} above. We begin by fixing some notation. We take $N$ to be a 
large positive number, and write $P=N^{1/3}$. Also, define
$$f(\alp)=\sum_{1\le x\le P}e(\alp x^3).$$
We define the minor arcs $\grn$ to be the set of points $\alp \in [0,1)$ 
satisfying the property that whenever $q\in \dbN$, and $q\alp$ differs from an 
integer by at most $P^{-9/4}$, then $q>P^{3/4}$. In addition, write $\grN=[0,1)
\setminus \grn$. By orthogonality, when $n\le N$, one has
$$R_{5,3}(n)=\int_0^1f(\alp)^5e(-n\alp)\,d\alp .$$

\par We put
$$E(n)=R_{5,3}(n)-\frac{\Gam({\textstyle{\frac{4}{3}}})^5}{\Gam({\textstyle{\frac
{5}{3}}})}\grS_{5,3}(n)n^{2/3},$$
and examine the value distribution of $E(n)$. Standard methods familiar to 
aficionados of the circle method confirm that, whenever $0<\del <\frac{1}{12}$ 
and $1\le n\le N$, one has
$$\int_\grN f(\alp)^5e(-n\alp)\,d\alp =\frac{\Gam({\textstyle{\frac{4}{3}}})^5}
{\Gam({\textstyle{\frac{5}{3}}})}\grS_{5,3}(n)n^{2/3}+O(N^{2/3-\del}).$$
Such a conclusion may be found in Lemma 2.1 of \cite{BW2009}, for example. In 
addition, as a consequence of Theorem 4.3 of \cite{Vau1997}, one has 
$\grS_{5,3}(n)\ll 1$. Thus, we may infer that when $h$ is positive and $0<\del 
<\frac{1}{12}$, one has
\begin{equation}\label{3.1}
\sum_{N/2<n\le N}|E(n)|^h\ll M_h(N)+O(N^{1+2h/3-h\del}),
\end{equation}
in which we have written
\begin{equation}\label{3.2}
M_h(N)=\sum_{N/2<n\le N}\Bigl| \int_\grn f(\alp)^5e(-n\alp)\,d\alp \Bigr|^h.
\end{equation}

\begin{theorem}\label{theorem3.1}
When $h\ge 2$, one has
$$M_h(N)\ll N^\eps(N^{\frac{11}{12}h}+N^{\frac{8}{9}h+\frac{2}{9}}+N^{\frac{3}{4}h+\frac{2}{3}}).
$$
\end{theorem}

\begin{proof} In order to bound the minor arc moments in (\ref{3.2}), we follow 
the procedure described at the beginning of \S3 of \cite{BW2009}. When $T$ is a 
positive number, let $\calZ_T(N)$ denote the set of natural numbers $n$ with 
$1\le n\le N$ for which one has the lower bound
$$\Bigl| \int_\grn f(\alp)^5e(-n\alp)\,d\alp\Bigr|>T.$$
Also, put $\calZ_T^*(N)=\calZ_T(N)\setminus \calZ_{2T}(N)$. We aim to bound 
$Z_T=\text{card}(\calZ_T^*(N))$. Define the complex numbers $\eta_n$ by putting 
$\eta_n=0$ for $n\not\in \calZ_T^*(N)$, and when $n\in \calZ_T^*(N)$ by means of 
the equation
$$\Bigl| \int_\grn f(\alp)^5e(-n\alp)\,d\alp\Bigr|=\eta_n \int_\grn f(\alp)^5
e(-n\alp)\,d\alp .$$
We note that $|\eta_n|=1$ when $\eta_n\ne 0$, and moreover
\begin{equation}\label{3.2w}
Z_TT\le \sum_{1\le n\le N}\eta_n\int_\grn f(\alp)^5e(-n\alp)\,d\alp =\int_\grn 
f(\alp)^5K_T(-\alp)\,d\alp ,
\end{equation}
where
$$K_T(\alp)=\sum_{1\le n\le N}\eta_ne(n\alp).$$

\par We first estimate the integral on the right hand side of (\ref{3.2w}) by 
means of Lemma \ref{2.1} above. On considering the underlying diophantine 
equations, one finds that
\begin{equation}\label{3.2a}
\int_0^1|f(\alp)^2K_T(\alp)^2|\,d\alp \ll P^\eps(PZ_T+P^{1/2}Z_T^{3/2}).
\end{equation}
From Hua's lemma (see Lemma 2.5 of \cite{Vau1997}), on the other hand, one has
$$\int_0^1|f(\alp)|^8\,d\alp \ll P^{5+\eps}.$$
An application of H\"older's inequality to (\ref{3.2w}) therefore reveals that
\begin{align*}
Z_TT&\le \Bigl( \int_0^1|f(\alp)^2K_T(\alp)^2|\,d\alp \Bigr)^{1/2}\Bigl( \int_0^1
|f(\alp)|^8\,d\alp \Bigr)^{1/2}\\
&\ll P^{5/2+\eps}(PZ_T+P^{1/2}Z_T^{3/2})^{1/2},
\end{align*}
whence
$$Z_T\ll P^\eps (P^6T^{-2}+P^{11}T^{-4}).$$
By dividing the range of summation into dyadic intervals, therefore, it follows 
that when $\nu>0$ and $2\le h\le 4$, one has
\begin{align}
\sum_{\substack{n\in \calZ_T^*(N)\\ P^{9/4}\le T\le P^{8/3+\nu}}}\Bigl| \int_\grn f(\alp)^5
e(-n\alp)\,d\alp \Bigr|^h&\ll P^\eps \left(P^6(P^{8/3+\nu})^{h-2}+P^{11}(P^{9/4})^{h-4}
\right)\notag \\
&\ll P^\eps (P^{\frac{8}{3}h+\frac{2}{3}+(h-2)\nu}+P^{\frac{9}{4}h+2}).\label{3.3}
\end{align}
Meanwhile, when $h>4$, one instead obtains
\begin{align}
\sum_{\substack{n\in \calZ_T^*(N)\\ P^{9/4}\le T\le P^{8/3+\nu}}}\Bigl| \int_\grn f(\alp)^5
e(-n\alp)\,d\alp \Bigr|^h&\ll P^\eps \left(P^6(P^{8/3+\nu})^{h-2}+P^{11}(P^{8/3+\nu}
)^{h-4}\right)\notag \\
&\ll P^{\frac{8}{3}h+\frac{2}{3}+(h-2)\nu+\eps}.\label{3.4}
\end{align}

\par Next we recall equation (3.8) of \cite{BW2009}, which supplies the estimate
$$\sum_{n\in \calZ_T^*(N)}\Bigl| \int_\grn f(\alp)^5e(-n\alp)\,d\alp \Bigr|^h\ll 
P^{13/2+\eps}T^{h-2}.$$
Again dividing the range of summation into dyadic intervals, we find that
\begin{align}
\sum_{\substack{n\in \calZ_T^*(N)\\ 0\le T\le P^{9/4}}}\Bigl| \int_\grn f(\alp)^5e(-n\alp)\,
d\alp \Bigr|^h&\ll P^{13/2+\eps}(P^{9/4})^{h-2}\notag \\
&\ll P^{\frac{9}{4}h+2+\eps}.\label{3.5}
\end{align}

\par Finally, we recall equations (3.12) and (3.13) of \cite{BW2009}, so that we
 have available the estimates
\begin{equation}\label{3.6}
\int_\grn f(\alp)^5e(-n\alp)\,d\alp \ll P^{11/4}
\end{equation}
and
\begin{equation}\label{3.7}
\sum_{\substack{n\in \calZ_T^*(N)\\ P^{8/3+\nu}\le T\le P^{11/4}}}\Bigl| \int_\grn f(\alp)^5e(-n
\alp)\,d\alp \Bigr|^h\ll P^{11/2}(P^{11/4})^{h-2}\ll P^{\frac{11}{4}h}.
\end{equation}

\par In order to confirm the conclusion of Theorem \ref{theorem3.1}, we have 
only to recall that $P=N^{1/3}$, note (\ref{3.6}), and collect together the 
estimates (\ref{3.3}), (\ref{3.4}), (\ref{3.5}) and (\ref{3.7}). In this way we 
conclude that
$$\sum_{n\in \calZ_T^*(N)}\Bigl| \int_\grn f(\alp)^5e(-n\alp)\,d\alp \Bigr|^h\ll 
N^\eps (N^{\frac{3}{4}h+\frac{2}{3}}+N^{\frac{8}{9}h+\frac{2}{9}+(h-2)\nu}+N^{\frac{11}{12}h}).$$
The desired conclusion then follows by taking $\nu$ sufficiently small, though 
positive.
\end{proof}

Returning to the relation (\ref{3.1}), we now take $\nu$ to be any positive 
number with $\nu<\frac{1}{5}$, and put $h=\frac{7}{2}-\nu$. Theorem 
\ref{theorem3.1} yields the estimate
\begin{align*}
M_h(N)&\ll N^{\frac{2}{3}h+1+\eps}(N^{\frac{1}{4}h-1}+N^{\frac{2}{9}h-\frac{7}{9}}+
N^{\frac{1}{12}h-\frac{1}{3}})\\
&\ll N^{\frac{2}{3}h+1-\nu/6}.
\end{align*}
On recalling (\ref{3.1}) and summing over dyadic intervals, we conclude that 
whenever $h$ is a positive number smaller than $\frac{7}{2}$, and 
$0<\del<\frac{1}{12}(7-2h)$, then
$$\sum_{1\le n\le N}|E(n)|^h\ll N^{2h/3+1-\del}.$$
This completes the proof of the first estimate of Theorem \ref{theorem1.7}.\par

The second estimate of Theorem \ref{theorem1.7} follows from the case $h=3$ of 
Theorem \ref{theorem3.1}, which delivers the bound
$$M_3(N)\ll N^{35/12+\eps}(N^{-1/6}+N^{-1/36}+1).$$
The desired conclusion therefore follows from (\ref{3.1}) by summing over dyadic
 intervals, since in that asymptotic relation one may take $\del$ to be any 
positive number smaller than $\frac{1}{12}$.

\section{A twelfth moment of cubic Weyl sums}
We turn our attention in this section to the problem of establishing the 
estimate for the twelfth moment of cubic Weyl sums claimed in Theorem 
\ref{theorem1.8}. Some preliminary manoeuvres are required to set the scene. 
Recall the notation and hypotheses of the statement of Theorem \ref{theorem1.8}.
 These hypotheses ensure that $\lam_1$, $\lam_2$ and $\lam_3$ are linearly 
independent, and thus there are non-zero integers $A$, $B$ and $C$, depending 
at most on $\bfc$ and $\bfd$, with the property that $(A,B,C)=1$ and 
$C\lam_3=A\lam_1+B\lam_2$. Write
$$\calF(\tet)=|F(\tet)^2h(\tet)^2|,$$
and then put
\begin{equation}\label{1.w}
\Tet_\bflam(P)=\int_0^1\int_0^1\calF(\lam_1)\calF(\lam_2)\calF(\lam_3)\,d\alp\,
d\bet .
\end{equation}
Then on making use of the periodicity of the integrand on the right hand side of
 (\ref{1.w}), and changing variables, one finds that
\begin{align*}
\Tet_\bflam(P)&=C^{-2}\int_0^C\int_0^C\calF(\lam_1)\calF(\lam_2)\calF(\lam_3)\,
d\alp \,d\bet \\ 
&=\int_0^1\int_0^1\calF(C\lam_1)\calF(C\lam_2)\calF(A\lam_1+B\lam_2)\,d\alp\,d\bet
\\
&\ll \int_0^1\int_0^1\calF(C\tet)\calF(C\phi)\calF(A\tet+B\phi)\,d\tet\,d\phi .
\end{align*}

\par Next, we write $R(n)$ for the number of representations of an integer $n$ 
in the shape
$$n=x_1^3-x_2^3+x_3^3-x_4^3,$$
with $P/2\le x_1,x_2\le P$ and $x_3,x_4\in \calA(P,R)$. Then, on considering the 
underlying diophantine equations, one finds that
$$\Tet_\bflam(P)\ll \sum_{|n_1|\le 2P^3}\sum_{|n_2|\le 2P^3}\sum_{|n_3|\le 2P^3}R(n_1)R(n_2)
R(n_3),$$
in which the summation is restricted by the conditions
\begin{equation}\label{4.1}
Cn_1=An_3\quad \text{and}\quad Cn_2=Bn_3.
\end{equation}
For suitable non-zero integers $a,b,c$, one finds that the integers $\bfn$ 
solving the system (\ref{4.1}) take the shape $\bfn=(ak,bk,ck)$, for some 
$k\in \dbZ$. We therefore find from H\"older's inequality that
$$\Tet_\bflam(P)\ll \sum_{\substack{|dk|\le 2P^3\\ (d=a,b,c)}}R(ak)R(bk)R(ck)\le 
(J_aJ_bJ_c)^{1/3},$$
where
$$J_d=\sum_{|dk|\le 2P^3}R(dk)^3\quad (d=a,b,c).$$

\par Define the exponential sum $g(\alp)=g(\alp;P)$ by
$$g(\alp;P)=\sum_{P/2<x\le P}e(\alp x^3).$$
Then, on considering the underlying diophantine equations, one finds that
$$J_d\le \sum_{|k|\le 2P^3}R(k)^3=\sum_{|k|\le 2P^3}\Bigl( \int_0^1|g(\tet)^2h(\tet)^2|
e(-\tet k)\,d\tet \Bigr)^3.$$
We therefore conclude at this point that
\begin{equation}\label{4.2}
\Tet_\bflam (P)\ll \sum_{|k|\le 2P^3}\Bigl( \int_0^1|g(\tet)^2h(\tet)^2|e(-\tet k)
\,d\tet \Bigr)^3.
\end{equation}
Define the sets of arcs $\grN$ and $\grn$ as in section 3. Then from Lemma 3.4 
of \cite{BKW2001}, one finds that
\begin{equation}\label{4.3a}
\int_\grN |g(\tet)^2h(\tet)^2|\,d\tet \ll P^{1+\eps}.
\end{equation}
We therefore deduce from (\ref{4.2}) that
\begin{equation}\label{4.3b}
\Tet_\bflam(P)\ll \sum_{|k|\le 2P^3}\Bigl| \int_\grn |g(\tet)^2h(\tet)^2|e(-\tet k)
\,d\tet \Bigr|^3+O(P^{6+\eps}).
\end{equation}

\par We now let $\calZ_T(P)$ denote the set of integers $k$ with $|k|\le 2P^3$ 
for which one has the lower bound
$$\Bigl| \int_\grn |g(\tet)^2h(\tet)^2|e(-\tet k)\,d\tet \Bigr| >T.$$
Also, we put $\calZ_T^*(P)=\calZ_T(P)\setminus \calZ_{2T}(P)$ and $Z_T=\text{card}
(\calZ_T^*(P))$. Define the complex numbers $\eta_k$ by putting $\eta_k=0$ for 
$k\not\in \calZ_T^*(P)$, and when $k\in \calZ_T^*(P)$ by means of the equation
$$\Bigl| \int_\grn |g(\tet)^2h(\tet)^2|e(-\tet k)\,d\tet \Bigr| =\eta_k
\int_\grn |g(\tet)^2h(\tet)^2|e(-\tet k)\,d\tet .$$
Again, we have $|\eta_k|=1$ whenever $\eta_k\ne 0$, and moreover
\begin{align}
Z_TT&\le \sum_{|k|\le 2P^3}\eta_k\int_\grn |g(\tet)^2h(\tet)^2|e(-\tet k)\,d\tet 
\notag \\
&=\int_\grn |g(\tet)^2h(\tet)^2|K_T(-\tet)\,d\tet ,\label{4.4}
\end{align}
where
$$K_T(\tet)=\sum_{|k|\le 2P^3}\eta_ke(k\tet).$$

\par We estimate the integral on the right hand side of (\ref{4.4}) through the 
medium of Lemma \ref{lemma2.1}. The estimate (\ref{3.2a}) again holds in the 
present context as a consequence of the latter lemma. Also, on considering the 
underlying diophantine equations, from Theorem 1.2 of \cite{Woo2000} one has
$$\int_0^1|g(\tet)^2h(\tet)^4|\,d\tet \ll P^{3+\xi+\eps},$$
where $\xi=\frac{1}{4}-\tau$. By applying H\"older's inequality to (\ref{4.4}) 
and considering the underlying diophantine equations, we therefore deduce that
\begin{align*}
Z_TT&\le \Bigl( \int_0^1|g(\tet)^2K_T(\tet)^2|\,d\tet \Bigr)^{1/2}\Bigl( \int_0^1
|g(\tet)^2h(\tet)^4|\,d\tet \Bigr)^{1/2}\\
&\ll P^\eps (PZ_T+P^{1/2}Z_T^{3/2})^{1/2}(P^{3+\xi})^{1/2},
\end{align*}
whence
$$Z_T\ll P^{4+\xi +\eps}T^{-2}+P^{7+2\xi+\eps}T^{-4}.$$
Let $\nu$ be a small positive number. Then by dividing the range of summation 
into dyadic intervals, we obtain the estimate
\begin{align}
\sum_{\substack{k\in \calZ_T^*(P)\\ P^{5/4+\xi/2}\le T\le P^{11/6+\nu}}}&\Bigl| \int_\grn 
|g(\tet)^2h(\tet)^2|e(-\tet k)\,d\tet \Bigr|^3 \notag\\
&\ll P^\eps \left(P^{4+\xi}(P^{\frac{11}{6}+\nu})+P^{7+2\xi}(P^{\frac{5}{4}+\frac{\xi}{2}})^{-1}
\right) \notag \\
&\ll P^{\frac{23}{4}+\frac{3\xi}{2}+\eps}.\label{4.5}
\end{align}

\par Next, we recall that a modified version of Weyl's inequality yields the 
bound
$$\sup_{\tet \in \grn}|g(\tet)|\ll P^{3/4+\eps}$$
(see, for example, Lemma 1 of \cite{Vau1986a}). Then by Schwarz's inequality and
 Parseval's identity, one obtains from (\ref{4.4}) the upper bound
\begin{align*}
Z_TT&\le \Bigl(\sup_{\tet \in \grn}|g(\tet)|\Bigr) \Bigl( \int_0^1|g(\tet)^2
h(\tet)^4|\,d\tet \Bigr)^{1/2}\Bigl( \int_0^1|K_T(\tet)|^2\,d\tet \Bigr)^{1/2}\\
&\ll P^{3/4+\eps}(P^{3+\xi})^{1/2}Z_T^{1/2},
\end{align*}
whence
$$Z_T\ll P^{9/2+\xi+\eps}T^{-2}.$$
Dividing the range of summation once again into dyadic intervals, we see now 
that
\begin{align}
\sum_{\substack{k\in \calZ_T^*(P)\\ 0\le T\le P^{5/4+\xi/2}}}\Bigl| \int_\grn |g(\tet)^2
h(\tet)^2|e(-\tet k)\,d\tet \Bigr|^3 &\ll P^{\frac{9}{2}+\xi+\eps}(P^{\frac{5}{4}+
\frac{\xi}{2}})\notag \\
&\ll P^{\frac{23}{4}+\frac{3\xi}{2}+\eps}.\label{4.6}
\end{align}

\par Finally, as a consequence of Hooley's work \cite{Hoo1978} on sums of four 
cubes, one has
$$\int_0^1|g(\tet)^2h(\tet)^2|e(-\tet k)\,d\tet \ll P^{11/6+\eps}$$
whenever $k\ne 0$ (see Lemma 2.1 of Parsell \cite{Par2000}). On recalling 
(\ref{4.3a}), it follows that when $k$ is non-zero, one has
$$\int_\grn |g(\tet)^2h(\tet)^2|e(-\tet k)\,d\tet \ll P^{11/6+\eps}+P^{1+\eps}\ll 
P^{11/6+\eps}.$$
When $k=0$, meanwhile, it follows from Hua's lemma (see Lemma 2.5 of 
\cite{Vau1997}) that
$$\int_0^1|g(\tet)^2h(\tet)^2|\,d\tet \ll P^{2+\eps}.$$
We therefore deduce that
\begin{align}
\sum_{\substack{k\in \calZ_T^*(P)\\ T>P^{11/6+\nu}}}\Bigl| \int_\grn |g(\tet)^2h(\tet)^2|
e(-\tet k)\,d\tet \Bigr|^3&=\Bigl( \int_\grn |g(\tet)^2h(\tet)^2|d\tet \Bigr)^3
\notag \\
&\ll (P^{2+\eps})^3.\label{4.7}
\end{align}

\par Combining (\ref{4.5}), (\ref{4.6}) and (\ref{4.7}), we find that
$$\sum_{k\in \calZ_T^*(P)}\Bigl| \int_\grn |g(\tet)^2h(\tet)^2|e(-\tet k)\,d\tet 
\Bigr|^3\ll P^{\frac{23}{4}+\frac{3\xi}{2}+\eps},$$
so that on summing over dyadic intervals, we deduce from (\ref{4.3b}) that
\begin{equation}\label{4.10}
\Tet_\bflam (P)\ll P^{\frac{23}{4}+\frac{3\xi}{2}+\eps}.
\end{equation}
The first estimate of Theorem \ref{theorem1.8} now follows on recalling the 
definition of $\Tet_\bflam(P)$.\par

In order to confirm the second estimate of Theorem \ref{theorem1.8}, we begin by
 making a dyadic dissection of the smooth Weyl sum $h(\alp)$. When $1\le R\le 
Q$, write $\calB(Q,R)=\calA(Q,R)\setminus \calA(Q/2,R)$, and define the 
exponential sum $H(\alp;Q)=H(\alp;Q,R)$ by 
$$H(\alp;Q,R)=\sum_{x\in \calB(Q,R)}e(\alp x^3).$$
We suppose throughout that $1\le R\le P^\eta$. Then on putting 
$L=[\frac{1}{2}\log P]$, we find that 
$$|h(\alp;P,R)|\le \sum_{l=0}^L|H(\alp;2^{-l}P)|+O(\sqrt{P}).$$
As a consequence of Hua's lemma (see Lemma 2.5 of \cite{Vau1997}), when 
$0\le l\le L$ and $1\le i\le 3$, one has
$$\int_0^1\int_0^1|H(\lam_i;2^{-l}P)|^4\,d\alp\,d\bet \le \int_0^1
|H(\tet;2^{-l}P)|^4\,d\tet \ll P^{2+\eps},$$
and when $0\le l,m\le L$ and $1\le i<j\le 3$, one has
\begin{align*}
\int_0^1\int_0^1|H(\lam_i;2^{-l}P)&H(\lam_j;2^{-m}P)|^4\,d\alp\,d\bet \\
&\le \int_0^1\int_0^1|H(\tet;2^{-l}P)H(\phi;2^{-m}P)|^4\,d\tet\,d\phi \ll P^{4+\eps}.
\end{align*}Thus we deduce that
\begin{equation}\label{4.11}
\int_0^1\int_0^1|h(\lam_1)h(\lam_2)h(\lam_3)|^4\,d\alp\,d\bet \ll P^\eps 
\Tet'_\bflam(P)+O(P^{6+\eps}),
\end{equation}
where
$$\Tet'_\bflam(P)=\max_{\sqrt{P}\le Q_1,Q_2,Q_3\le P}\int_0^1\int_0^1|H(\lam_1;Q_1)
H(\lam_2;Q_2)H(\lam_3;Q_3)|^4\,d\alp\,d\bet .$$

\par From here, the argument applied above leading from (\ref{1.w}) to 
(\ref{4.2}) may be applied, mutatis mutandis, and thereby we establish via 
H\"older's inequality that
$$\Tet'_\bflam (P)\ll \max_{\sqrt{P}\le Q\le P}\sum_{|k|\le 2Q^3}\Bigl( \int_0^1
|H(\tet;Q)|^4e(-\tet k)\,d\tet \Bigr)^3.$$
Hence, on considering the underlying diophantine equations, one finds that
$$\Tet'_\bflam (P)\ll \max_{\sqrt{P}\le Q\le P}\sum_{|k|\le 2Q^3}\Bigl( \int_0^1
|g(\tet;Q)^2h(\tet;Q,Q^{3\eta})^2|e(-\tet k)\,d\tet \Bigr)^3.$$
A comparison between the sum on the right hand side of this estimate, with that 
on the right hand side of (\ref{4.2}), reveals that the argument employed above 
to deliver (\ref{4.10}) in this instance shows that
$$\Tet'_\bflam (P)\ll  \max_{\sqrt{P}\le Q\le P}Q^{\frac{23}{4}+\frac{3\xi}{2}+\eps} 
\le P^{\frac{23}{4}+\frac{3\xi}{2}+\eps}.$$
The second estimate of Theorem \ref{theorem1.8} consequently follows from 
(\ref{4.11}).

\bibliographystyle{amsbracket}
\providecommand{\bysame}{\leavevmode\hbox to3em{\hrulefill}\thinspace}

\end{document}